\documentclass[a4paper,11pt,english]{article}

\usepackage{authblk}

\usepackage{ulem}

	\usepackage{lmodern}				

	\usepackage[T1]{fontenc}			

	\usepackage[utf8x]{inputenc}		

\usepackage{graphicx} 

	\usepackage{babel}					

	\usepackage{amsmath}					

	\usepackage{amsthm}

	\usepackage{amssymb}					

\RequirePackage{bbm}					

\RequirePackage{stmaryrd}

	\usepackage{pstricks}

	\usepackage{hyperref}

	\usepackage{cleveref}

	\usepackage{geometry}				
	
\usepackage{algorithm}
\usepackage{algorithmic}

\newcommand{\ensnombre}[1]{\mathbb{#1}}%

\newcommand{\R}{\ensnombre{R}}

\newcommand{\ind}{\mathbbm{1}}

\renewcommand{\P}{\mathbb{P}}

\newcommand{\E}{\mathbb{E}}

\def \Var{\hbox{{\textrm{Var}}}}

\theoremstyle{plain}

\newtheorem{thm}{Theorem}[section]

\newtheorem{cor}[thm]{Corollary}

\newtheorem{lem}[thm]{Lemma}

\newtheorem{prop}[thm]{Proposition}

\theoremstyle{remark}

\newtheorem{rem}[thm]{Remark}

\renewcommand{\geq}{\geqslant}
\renewcommand{\leq}{\leqslant}

\usepackage{setspace}
\usepackage{comment}

\title{Asymptotic efficiency for Sobol' and Cramér-von Mises indices under two designs of experiments}

\author[1,2]{Thierry Klein}
\author[2,3]{Agn\`es Lagnoux} 
\author[4,5]{Thi Mong Ngoc Nguyen}
\author[1]{Paul Rochet}

\affil[1]{F\'ed\'eration ENAC ISAE-SUPAERO ONERA, Universit\'e de Toulouse, France.}
\affil[2]{Institut de Math\'ematiques de Toulouse; UMR5219. Universit\'e de Toulouse; CNRS.}
\affil[3]{UT2J, F-31058 Toulouse, France.}
\affil[4]{University of Science, VNUHCM, Ho Chi Minh City, Vietnam.}
\affil[5]{Vietnam National University, Ho Chi Minh City, Vietnam.}
\begin{document}

\maketitle


\begin{abstract} A variety of indices aim to quantify the impact of input variables on a response, typically the output from a complex computer code or black-box model. Most commonly used, the Sobol' index typically measures the influence of some inputs from an explained variance perspective. 
However, some situations may require a more targeted analysis of some inputs influence. With no prior information, distribution-based measures appear to be appealing. In this purpose, so-called Cramér-von Mises indices (and their generalization) have been proposed in the literature, defined as an excess probability integrated over the output distribution that aim to reflect influence on the whole distribution of the output rather than on the variance solely. Inference of these various indices has remained a challenging topic especially in presence of many inputs. While several Sobol' indices estimators are known to be optimal under regularity conditions,  the issue of asymptotic efficiency for Cramér-von Mises indices
has  been unaddressed in the literature so far. For these indices, we derive in this paper the efficiency bounds  and discuss the known methods to achieve such optimal bounds. Two estimation contexts are considered: the so-called Pick-Freeze scheme and the Given-Data setting, for which the estimation is produced from a unique input-output sample.


\end{abstract}



\section{Introduction}\label{sec:intro}

The use of complex computer models for the simulation and the analysis of natural systems from  physics, engineering, and other fields is by now routine. These models usually depend on many input variables and it is thus crucial to understand which input parameter - or which set of input parameters - have an influence on the output. This is the aim of sensitivity analysis which has become an essential tool for system modeling and policy support (see, e.g., \cite{razavi:hal-03518573}). Global sensitivity analysis methods consider the input vector as random and propose a measure of  influence of each subset of its components on the output. We refer to the seminal book \cite{saltelli-sensitivity} for an overview on global sensitivity analysis or to \cite{da2021basics} for a synthesis of recent trends in the field.

\medskip

More precisely, for the output $Y$ of a computer code $Y=G(X,W)$ where the inputs $X$ and $W$ are assumed to be mutually independent, sensitivity indices aim at quantifying the relative influence of one input on the output $Y$.
Among the different measures of global sensitivity analysis, variance-based measures are probably the most commonly used. In particular, Sobol' indices (introduced in \cite{pearson1915partial} and later revisited in the framework of sensitivity analysis in \cite{sobol1993,sobol2001global}) are the most popular and appealing in practice.  They were then generalized in several ways but can in most cases be written in  a synthetic closed form as follows

\begin{equation}\label{eq:formuleG} SI = \frac{\psi - \phi_2}{\phi_1 - \phi_2} \end{equation}
where $\psi$ is the only quantity depending on $X$ and $\phi_1$ and $\phi_2$ are normalising factors ensuring that the sensitivity index lies in $[0,1]$.

\medskip

Since in practice computing explicitly the theoretical value of $SI$ is out of reach, one of the main tasks in sensitivity analysis is to provide estimators, with guaranteed asymptotic properties such as consistency, rate of convergence (central limit theorem)... In the recent years, a myriad of different estimators has been proposed, see \cite[Chapter 4]{da2021basics} for a complete review. To compare these different estimators, it is then relevant to define a notion of ``optimality'' using a concept similar to the Cramér-Rao bound in parametric statistics.  
In this framework, optimality is assessed via the notion of \textit{asymptotic efficiency} and \textit{efficient influence functions} introduced in the seminal works \cite{hajek1970characterization, inagaki1970limiting} in a parametric setting and further extended to semi and non-parametric models in \cite {levit1976efficiency, beran1977robust, begun1983information} (see also \cite{bickel1993efficient, van2000asymptotic} for an extensive description of the theory of asymptotic efficiency). 

\medskip 

We tackle the issue of asymptotic efficiency in the two main frameworks related to Sobol' index inference, namely the \textit{Pick-Freeze} and the \textit{Given-Data} settings. When the context allows it, using the particular Pick-Freeze setting where two draws $Y$ and $Y'$ of the response are available for each realization of the input $X$, considerably simplifies the estimation process. The key to its success is to exploit a secondary expression of the parameter of interest $\psi $ 
making a natural empirical estimator of $ \psi $ available. The underlying model in this setting is the set of exchangeable bi-variate distributions.

\medskip

In most situations however, the practitioner cannot afford the luxury of choosing the input's values when generating the data. The most common scenario is to deal with independent and identically distributed (i.i.d.) copies of $(X,Y)$ where a duplicate draw $Y'$ is no longer available, which constitutes a particular case of Given-Data. In the non-parametric model on the distribution of the pair $(X,Y)$, 
asymptotically efficient estimators of the Sobol' index have been proposed in the literature, although they usually require strong assumptions such as a low-dimensional input or an extensively smooth non-parametric regression function $x \mapsto \mathbb E[Y | X = x]$. In practice, building an asymptotically efficient estimator that lives up to its theoretical properties on numerical simulations (e.g.~for high-dimensional inputs) remains somewhat of an open problem \cite{GKLPdV21}.

\medskip

When dealing with Cramér-von Mises indices or their generalization (presented in more details in Section \ref{sec:indices}),  their exact computation is even more challenging and the need of estimation is increased. Here again, one may propose estimators in each of the two settings discussed above: the Pick-Freeze one and the Given-Data one. Comparing the asymptotic performance of such estimators appears again to be essential.

\medskip

The article is organized as follows. In Section \ref{sec:indices}, we recall the definition of the Sobol' indices, the Cramér-von Mises indices, and their generalization together with the aim of this paper. We also recall what is meant by Pick-Freeze and Given-Data settings. Some intermediate discussions and results on asymptotic efficiency and efficient influence functions are given in Section \ref{sec:ass_eff}.  Finally, Sections \ref{sec:PF} and \ref{sec:given_data} are devoted to the characterization of the efficient influence functions for the Sobol' indices (Sections \ref{ssec:PF_sobol} and \ref{ssec:given_data_sobol}), for the Cramér-von Mises indices (Sections \ref{ssec:PF_CVM_mu_Y} and \ref{ssec:given_data_CVM_mu_Y}), for their generalized version (Sections \ref{ssec:PF_CVM_mu_known}-\ref{ssec:PF_CVM_mu_unknown} and \ref{ssec:given_data_CVM_mu_known}-\ref{ssec:given_data_CVM_mu_unknown}), and for the universal indices (Sections \ref{ssec:PF_univ} and  \ref{ssec:gicen_data_univ}) in the Pick-Freeze and in the Given-Data settings respectively. The paper concludes with a discussion on the known methods that achieves such optimality.


\section{Standard indices in sensitivity analysis}\label{sec:indices}


As mentioned in the Introduction, we consider the following model 
\[
Y = G(X,W)
\] 
for some measurable function $G$, where $Y$ is a real-valued square integrable output, $X \in \mathbb R^d$ is a vector-valued input, and $W$ a random term independent of $X$. 

\paragraph{Sobol' indices}

In the context of global sensibility analysis, the Sobol' index quantifies the influence of the input $X$ on $Y$ through the proportion of explained variance, using the conditional expectation $m(X) := \E[Y|X]$ as the best approximation (in the $L^2$ sense) of the response $Y$ by the input $X$. The index is thus defined as the variance ratio
\begin{equation}\label{def:sobol}
     S = \frac{\Var(m(X))}{\Var(Y)} \in [0,1],
\end{equation}
with a value close to one indicating a strong influence of the input $X$ on the output $Y$. For inference purposes, it is convenient to decompose the value of $S$ as a function of three somewhat simple parameters that can be handled separately. Specifically, we write
\begin{equation}\label{eq:formule} SI = \frac{\psi - \phi_2}{\phi_1 - \phi_2} \end{equation}
where $\psi = \mathbb E [m^2(X)] = \mathbb E [m(X) Y]$, while $\phi_1 =\mathbb E [Y^2] $ and $\phi_2=\mathbb E [Y]^2$ are the first two moments of the output. The parameter $\psi$ is the most relevant for our purposes as it is the only term involving the input $X$. It is also more challenging to estimate due to the presence of the unknown regression function $m$. The other two parameters $\phi_1$ and $\phi_2$ only depend on the output and their role is essentially to normalize the index value in the interval $[0,1]$. 

\medskip

The estimation of \textit{Sobol' indices} is most commonly performed in one of two main sampling settings, usually referred to as Pick-Freeze and Given-Data in the global sensitivity analysis literature. When possible, the Pick-Freeze setting where two draws $Y$ and $Y'$ of the response for each realization of the input $X$ allows an alternative expression of the parameter $\psi$ as
$$ \psi = \mathbb E [Y Y'], $$
for which a natural empirical estimator can be built from an i.i.d.\ sample $(Y_i, Y'_i)$,  $i=1,...,n$, known to be asymptotically efficient \cite{janon2012asymptotic}. Observe that the pair $(Y, Y')$ is exchangeable (i.e.~$(Y, Y')$ and $(Y', Y)$ are identically distributed).
By de Finetti's theorem \cite{de2017theory}, the model in the Pick-Freeze setting thus coincides with the set of exchangeable bi-variate distributions.  
Meanwhile, optimal estimation of the output's first two moments $\phi_1$ and $\phi_2$ is easily achieved empirically using the whole sample $(Y_i, Y'_i), i=1,...,n$. 

\medskip

Things are different in the Given-Data setting where the data consist of an i.i.d.\ sample from the pair $(X, Y)$. Here, no simple alternative expression of $\psi $ can be used to avoid dealing with the conditional expectation $m(X) = \mathbb E[ Y \, | \, X]$. While the calculation of the efficiency bound for $\psi$ is more technical in this context, its expression was first provided in \cite{doksum1995nonparametric} for a continuous input $X$, from the derivation of the efficient influence function associated to the parameter. Since then, several estimation methods have been proved to be asymptotically efficient in the Given-Data setting. For first-order inputs, asymptotically efficient estimators are proposed in \cite{DaVeiga13} and \cite{klein2024efficiency} while kernel estimation methods allow to provide an optimal estimator in higher dimensions under smoothness regularity conditions \cite{doksum1995nonparametric,GKLPdV21}. Here again, the other two parameters $\phi_1$ and $\phi_2$ are easily handled empirically. 

\paragraph{Cramér-von Mises indices}

The \textit{Cramér-von Mises index}, introduced in \cite{dette2013copula} and later studied in \cite{GKL18}, provides an alternative to the traditional variance-based approach in global sensitivity analysis. It can deal naturally with a multi-dimensional output $Y \in \mathbb R^k$ by replacing the conditional expectation by a conditional cumulative distribution function (c.d.f.)
$$ F(z \, | \, X) = \mathbb P(Y \leq z \, | \, X) \, , \, z \in \mathbb R^k.  $$
Here $\{Y\leqslant  z\}$ means that $\{Y_1\leqslant  z_1, \ldots, Y_k\leqslant  z_k\}$.
Letting  $F_Y(z) = \mathbb P(Y \leq z)$ for $z \in \mathbb R^k$ denote the c.d.f.\ of $Y$, the Cramér-von Mises index is then defined by
\begin{equation}\label{def:CVM} 
SI = \frac{ \int \Var (F(z \, | \, X)) \, d F_Y(z)}{\int F_Y(z) ( 1 - F_Y(z)) \, dF_Y(z) } = \frac{\psi - \phi_2}{\phi_1 - \phi_2},
\end{equation}
which fits the decomposition of Equation \eqref{eq:formule} with the three parameters
\begin{equation}\label{eq:integrated}  \psi = \int F(z \, | \, X)^2 \, d F_Y(z),  \quad \phi_1 = \int F_Y(z) \, d F_Y(z), \quad \text{ and } \quad \phi_2 = \int F_Y(z)^2 \, d F_Y(z). \end{equation}
A natural extension of this idea is to integrate the c.d.f.\ over an arbitrary probability measure $Q$ on $\mathbb R^k$, as a way to target particular values of the response. Thus the generalized Cramér-von Mises index has a similar expression as in \eqref{def:CVM} with the three parameters integrated over $Q$,
\begin{equation}\label{def:CVMmu}
     \psi = \int F(z \, | \, X)^2 \, d Q(z),   \quad \phi_1 = \int F_Y(z) \, d Q(z), \quad \text{ and } \quad \phi_2 = \int F_Y(z)^2 \, d Q(z).
\end{equation}

Inference of generalized Cramér-von Mises indices can be implemented in both the Pick-Freeze and Given-Data settings. An estimator of the initial Cramér-von Mises index is provided in \cite{GKL18} in the Pick-Freeze setting and in \cite{Chatterjee2019,FKL21} in the Given-Data setting, although the question of its optimality is not addressed. In Sections \ref{sec:PF} and \ref{sec:given_data}, we derive the efficient influence function and the asymptotic efficiency bound for original and generalized Cramér-von Mises indices under both sampling schemes after presenting some preliminary results on asymptotic efficiency.

\paragraph{Universal sensitivity indices} 

The previous generalized Cramér-von Mises indices can also be extended in another direction as done in \cite{FKL21}. More precisely, consider an output $Y$ lyning in a metric space $\mathcal Y$ and a family of test functions  parameterized by an element $z$ in some probability space $\mathcal Z$ and defined by 
 \[
\begin{matrix}
& \mathcal Z \times \mathcal Y & \to & \R\\
& (z,y) & \mapsto & T_z(y).\\
\end{matrix}
\]
Replacing the indicator function $\ind_{\{Y\leqslant z\}}$ by the test function $T_z(y)$ in the generalized Cramér-von Mises definition \ref{def:CVMmu}, we introduce the universal sensitivity index that has a similar expression as in \eqref{def:CVM} with the three parameters integrated over the distribution $Q$ on $\mathcal Z$,
\begin{equation}\label{def:univ}
     \psi = \int \E[T_z(Y) \, | \, X]^2 \, d Q(z), \; \phi_1 = \int \E[T_z(Y)^2] \, d Q(z),  \text{ and }  \phi_2 = \int \E[T_z(Y)]^2 \, d Q(z).
\end{equation}

\medskip

The universality is twofold. First, it allows to consider more general relevant  indices. 
Secondly, this definition encompasses, as particular cases, the classical sensitivity indices. 
\begin{itemize}
\item The so-called classical Sobol' index corresponds to $\mathcal Y = \mathcal Z = \R$ and the indentity test function $T_z=\operatorname{Id}$ (see \cite{sobol1993,sobol2001global}). The parameterized test functions do not depend on $z$.  
\item For $\mathcal Y = \mathcal Z = \R^k$ and $T_z(y)=\ind_{\{y\leqslant z\}}$,  one recovers the index based on the Cram\'er-von-Mises distance defined in \cite{GKL18} and recalled in \eqref{def:CVM}. 
\item For $\mathcal Y=\mathcal M$ ($\mathcal M$ being a manifold), $\mathcal Z=\mathcal M^2$,  and $T_z(y)=\ind_{\{y\in \widetilde B(z_1,z_2)\}}$, where $\widetilde B(z_1,z_2)$ stands for the ball whose center is the middle point between $z_1$ and $z_2$ with radius $\overline{z_1z_2}/2$, one recovers the index defined in \cite{FGM2017}. 
\item Consider the Wasserstein space $\mathcal W_2 (\R)$ i.e.\ the space of all measures defined on $\R$  endowed with the $2$-Wasserstein distance $W_2$ with finite $2$-moments. Let that $\mathcal Y = \mathcal W_2 (\R)$, $\mathcal Z = \mathcal W_2 (\R)^2$ and $T_{(F_1,F_2)}(F)=\ind_{W_2 (F_1,F)\leqslant W_2 (F_1,F_2)}$ for $F$, $F_1$, and $F_2$ three elements of  $\mathcal W_2 (\R)$. The corresponding indices introduced in \cite{FKL21} allow to perform sensitivity analysis for stochastic computer codes. 
\end{itemize}

\medskip

The aim of this work is to provide the influence efficient function for the classical sensitivity indices under two classical design of experiments (Pick-Freeze and Given-Data). We state and prove our results for the most common indices (Sobol' and Cramér-von Mises indices) and give the  results for generalized Cramér-von Mises and the universal sensitivity indices. In order to be self-contained, we start by recalling in Section \ref{sec:ass_eff} some basic facts about asymptotic efficiency.

\section{Background on asymptotic efficiency}\label{sec:ass_eff}


In a statistical model $\mathcal P$, the asymptotic efficiency bound for estimating a real parameter $\psi:\mathcal P \to \mathbb R$ is determined by studying the behavior of the parameter $\psi(P_t)$ as $P_t$ smoothly approaches a target distribution $P \in \mathcal P$ as $t \to 0^+$. Based on the idea that a parameter $\psi$ whose value hardly varies in a neighborhood of $P$ should be easier to estimate, 
the derivative of $t \to \psi(P_t)$ at $t=0^+$, if it exists, should provide some information on the quality that can be expected from an estimator of $\psi(P)$. For this, the parameter $\psi$ must be \textit{differentiable} in the sense that the derivative of $t \mapsto \psi(P_t)$ at $t = 0^+$ can be expressed as a continuous linear functional of a \textit{score function} $g$ of a smooth submodel $(P_t)_{t \in [0, \varepsilon)}$. Since all score functions are square-integrable with respect to $P$, the Riesz representation theorem states that there exists $\dot{\psi} \in L^2(P)$ such that
\begin{equation}\label{eq:infl} \lim_{t \to 0^+} \frac{\psi(P_t) - \psi(P)} t = \int \dot{\psi} g \, dP.    
\end{equation}
In asymptotic efficiency, a representation $\dot \psi$ is called an \textit{influence function}. When restricted to a one-dimensional parametric model $\mathcal P_g$ with score $g$, the well-known Cram\'er-Rao bound 
$$ B(\mathcal P_g) = \frac{ \int \dot{\psi} g \, dP }{\int g^2 dP}   $$
provides a lower bound for 
the asymptotic variance of a regular estimator of $\psi(P)$. In a larger model $\mathcal P$ where lesser information on the true distribution is available, a lower bound can not be smaller than that of any one-dimensional submodel $\mathcal P_g \subset \mathcal P$. Thus, the least favorable way to approach $P$ in the model $\mathcal P$, or equivalently the least favorable score function, provides a lower bound for the asymptotic variance of a regular estimator by 
$$ B(\mathcal P) =  \sup_{g \in \dot{\mathcal P}} \frac{ \int \dot{\psi} g \, dP}{\int g^2 dP}   $$
where $\dot{\mathcal P} \subset L^2(P)$ denotes the set of all score functions in $\mathcal P$ at $P$ (the tangent set). If $\dot{\mathcal P}$ is a closed linear subset of $L^2(P)$, it is easy to show that 
\begin{equation}\label{eq:score} \sup_{g \in \dot{\mathcal P}} \frac{ \int \dot{\psi} g \, dP}{\int g^2 dP}  = \int \widetilde \psi^2 dP 
\end{equation}
where $\widetilde \psi \in \dot{\mathcal P}$ is the unique influence function in the tangent set, called the \textit{efficient influence function}. The tangent set $\dot{\mathcal P}$ and the efficient influence function $\widetilde \psi$ are always implicitly calculated at the true distribution $P$ without specifying it. Similarly, it is always understood that a submodel $(P_t)_{t \in [0, \varepsilon)}$ is differentiable in quadratic mean with score $g$ at $t=0$ and $P_0 = P$. We refer to \cite{van2000asymptotic} for a detailed description of the theory. 

\medskip

Given an i.i.d.\ sample $Z_1,...,Z_n$ drawn from $P$, a necessary and sufficient condition for a (regular) estimator $\hat \psi_n$ of $\psi(P)$ to be asymptotically efficient is given by the following fundamental lemma. 

\begin{lem}[Lemma 25.23, \cite{van2000asymptotic}]\label{lem:eff_expansion}
A regular estimator $\hat \psi_n$ is asymptotically efficient at $P$ if and only if  the following expansion holds 
\begin{align}\label{eq:ass_eff}
     \hat \psi_n  =  \psi(P) + \frac 1 n \sum_{i=1}^n \widetilde \psi(Z_i) + o_{\P}( n^{-1/2}).  
\end{align}  

\end{lem}

While seemingly a condition on the limit distribution of an estimator, this result shows that asymptotic efficiency is really a property of convergence in probability. In particular,
coordinate-wise asymptotically efficient estimators $\hat \psi^1_n, \cdots, \psi_n^k $ are necessarily jointly asymptotically efficient, since the efficient influence function of the vector parameter $\Psi =(\psi^1, \cdots, \psi^k) $ at $P$ is the vector $\widetilde \Psi = (\widetilde\psi^1, \cdots, \widetilde\psi^k)$ of efficient influence functions. Moreover, asymptotic efficiency is stable through smooth transformations: if $h$ is a differentiable function from $\mathbb R^k$ to $\mathbb R$, the efficient influence function of the parameter $h \circ \Psi(P)$ can be identified from the definition by
\begin{align}\label{eq:EIF_compo} \lim_{t \to 0^+} \frac{h \circ \Psi(P_t) - h\circ \Psi (P)} t =  \int \bigl( \nabla h(\Psi(P))^\top \widetilde \Psi  \bigr) g \,  dP ,  
\end{align}
where $\nabla$ denotes the gradient operator. As a direct consequence of the Delta method, the estimator $h(\hat \psi_n^1, \cdots, \hat \psi_n^k) $ thus satisfies the condition of Lemma \ref{lem:eff_expansion},
\begin{align}\label{eq:estim_compo}
h(\hat \psi_n^1, \cdots, \hat \psi_n^k)  =  h \circ \psi (P) + \frac 1 n \sum_{i=1}^n \nabla h(\Psi(P))^\top \widetilde \Psi(Z_i) + o_\P(n^{-1/2}),      
\end{align}
proving it is asymptotically efficient. 

\medskip

In the next two sections, we determine the efficient influence functions and the efficiency bound for various sensitivity indices in the Pick-Freeze and in the Given-Data frameworks. The strategy to proceed always follows the same three-step guideline.
\begin{enumerate}
    \item Determine the model $\mathcal P$ and the tangent set $\dot{\mathcal P}$. In all cases discussed in this paper, $\dot{\mathcal P}$ is a closed linear subspace of $L^2(P)$. 
    \item Exhibit an influence function $\dot{\psi}$ of the parameter $\psi(P)$ using the Riesz representation of the derivative given in Equation~\eqref{eq:infl}.
    \item Find the efficient influence function $\widetilde \psi$ as the only representative in the tangent space.
\end{enumerate}

To find an influence function from Equation~\eqref{eq:infl}, one should expect to have to consider all score functions $g \in \dot{\mathcal P}$ and all arbitrary submodels. However, the formulation of the efficiency bound in Equation~\eqref{eq:score} suggests that considering a dense linear subset of $\dot{\mathcal P}$ is actually sufficient since it leads to the same value of the supremum. A good practical choice for such subset is to consider score functions $g$ that are bounded $P$-almost-surely. Then, a simple submodel $(P_t)_{t \in [0, \varepsilon)} $ with score $g$ is obtained by setting $P_t = (1+tg) P$, guaranteed to be a probability measure for $t$ sufficiently small. In fact, to exhibit an influence function $\dot{\psi}$ of a differentiable parameter $\psi(P)$, it is sufficient to verify that Equation~\eqref{eq:infl} holds for the submodels $P_t = (1+tg)P$ with $g \in \dot{\mathcal P}$ bounded $P$-almost-surely. We use this argument in most of the proofs.

\medskip

We start with a simple lemma that determines the tangent set of tensor product statistical models. 
For any measurable space $E$ (implicitly equipped with a $\sigma$-algebra $\mathcal A$), we denote by $\mathcal P(E)$ the set of probability measures on $E$. A statistical model $\mathcal P$ on $E$ is a subset of $\mathcal P(E)$. 

\begin{lem}\label{lem:tangent_prod} Let $E$ anf $F$ be two measurable spaces and $\mathcal P \subseteq \mathcal P(E)$ and  $\mathcal Q \subseteq \mathcal P(F)$ two statistical models 
with tangent sets $\dot{\mathcal P}$ and $\dot{\mathcal Q}$ at $P \in \mathcal P$ and $Q \in \mathcal Q$ respectively. The tangent set of the tensor product model
$$ \mathcal P \!\otimes\! \mathcal Q : = \{ P \!\otimes\! Q : P \in \mathcal P, \, Q \in \mathcal Q \} \subset \mathcal P(E \times F) $$
at $P \!\otimes\! Q$ is the direct sum of tangent spaces
$$ \dot{\mathcal P} \!\oplus\! \dot{\mathcal Q} := \{ g + h : g \in \dot{\mathcal P}, \,  h \in \dot{\mathcal Q} \} $$
where $g + h : E \times F \to \mathbb R$ is defined canonically by $(g+h)(y,z) = g(y) + h(z)$.
\end{lem}

\begin{proof}[Proof of Lemma \ref{lem:tangent_prod}]

Let $P_t = (1+tg) P$ and $Q_t = (1+th) Q$ with $g \in \dot{\mathcal P}$ and $h \in \dot{\mathcal Q}$. Since
$$ P_t \!\otimes\! Q_t = (1 + t(g+h) + t^2 gh) (P \!\otimes\! Q) ,  $$
where $gh$ is square-integrable with respect to $P \!\otimes\! Q$, the submodel $(P_t \!\otimes\! Q_t)_t$ has score $ g+ h $.
\end{proof}

The next proposition provides an expression of the efficient influence function for integrated parameters of the form $\psi = \int \psi_z d Q(z)$ such as those of Equation \eqref{eq:integrated}. This result will be useful to derive the efficiency bound of generalized Cramér-von Mises indices, whether $Q$ is known or unknown. Corollary \ref{cor:integral_influence} will handle the regular Cramér-von Mises index where $Q$ equals the distribution of the output $Y$.


\begin{prop}\label{prop:integral_influence} Let $\mathcal P$ and $\mathcal Q$ be two statistical models on respective measurable spaces $E$ and $F$. For all $z \in F$, let $\psi_z: \mathcal P \to \mathbb R$ be a real-valued parameter on $\mathcal P$ with efficient influence function $\widetilde \psi_z : E \to \mathbb R$ at $P$. Assume that the maps $ P \mapsto \psi_z(P)$ are  uniformly bounded for all $z \in F$ and verify, for all $g \in \dot{\mathcal P}$ bounded and $P_t = (1+tg) P$,
\begin{equation}\label{eq:condi} \sup_{z \in F} \, \Big| \psi_z(P_t) - \psi_z(P) - t \int \widetilde \psi_z(y) g(y) dP(y) \Big| = o(t). 
\end{equation}
\begin{enumerate}
    \item For a fixed $Q \in \mathcal Q$, the parameter 
$$ \psi^Q(P) := \int \psi_z(P) dQ(z), \quad P \in \mathcal P  $$
is well defined on $\mathcal P$ and has efficient influence function at $P$ given by 
$$\widetilde \psi^Q(y) = \int \widetilde \psi_z(y) dQ(z), \quad y \in E. $$  
    \item The parameter $\Psi (P \!\otimes\! Q)=  \psi^Q(P) , P \in \mathcal P , Q \in  \mathcal Q $ has efficient influence function at $P \!\otimes\! Q$:
    $$ \widetilde \Psi ( y,z ) = \widetilde \psi^Q(y) + \xi_{P,Q}(z), \quad (y, z) \in E \times F, $$
    with $\xi_{P,Q}(z)$ the orthogonal projection of $z \mapsto \psi_z(P)$ onto the linear span of $\dot{\mathcal Q}$ in $L^2(Q)$.
\end{enumerate}
\end{prop}

\begin{rem}
The result presented in item 1 can be deduced from item 2 with $\mathcal Q=\{Q\}$. 
\end{rem}

\begin{proof}[Proof of Proposition \ref{prop:integral_influence}] By Assumption \eqref{eq:condi}, for any submodel $(P_t)_{t >0}$ with score $g$ at $P_0 = P$,
$$ \frac{\psi^Q(P_t) - \psi^Q(P)} t =  \int \bigg( \int \widetilde \psi_z(y) dQ(z) \bigg) g(y) dP(y) + o(1). $$

Clearly, $\widetilde \psi^Q : y \mapsto \int \widetilde \psi_z(y) dQ(z)$ lies in $\operatorname{cl} ( \dot{\mathcal P} )$ by convexity, it is therefore the efficient influence function of $\psi^Q$. 

For the parameter $\Psi$, as discussed previously, consider a submodel $(P_t \!\otimes\! Q_t )_{t >0} $ on $\mathcal P \!\otimes\! \mathcal Q$ with $Q_t = (1+th) Q$. It has score $h + g$ at $P \!\otimes\! Q$ by Lemma \ref{lem:tangent_prod} and
\begin{align*} \Psi(P_t \!\otimes\! Q_t) & = \int \psi_z(P_t) dQ(z) + t \int \psi_z(P_t) h(z) dQ(z)\\
& = \Psi(P \!\otimes\! Q) + t \int \widetilde \psi^Q(y) g(y) d P(y) + t \int \psi_z(P) h(z) dQ(z) + o(t) 
\end{align*}
using assumption \eqref{eq:condi}.
Since $ \int g dP = \int h dQ = 0$, we may write
$$  \lim_{t \to 0^+} \frac{\Psi(P_t \!\otimes\! Q_t) - \Psi(P \!\otimes\! Q)} t = \int \big( \widetilde \psi^Q(y) + \psi_z(P) \big) \big( g(y) + h(z)\big) d (P \!\otimes\! Q)(y,z).$$
The function $(y,z) \mapsto \widetilde \psi^Q(y) + \psi_z(P) $ is thus an influence function. As $\widetilde \psi^Q$ lies in $\dot{\mathcal P}$, the result follows. 
\end{proof}

\begin{cor}\label{cor:integral_influence} Let $\mathcal P$ be a statistical model on a measurable space $E$ and $(\psi_z, z \in E)$ be a collection of real-valued parameters on $\mathcal P$. Suppose that Assumption \eqref{eq:condi} in Proposition \ref{prop:integral_influence} holds. 
Then, the parameter $\Psi(P) = \psi^P(P)$, $P \in \mathcal P$ has efficient influence function at $P$:
    $$ \widetilde \Psi(y) = \widetilde \psi^P(y) + \xi_{P,P}(y) , \quad y \in E  $$
    where the notation $\psi^P(P)$, $\widetilde \psi^P$, and $\xi_{P,P}$ have been introduced in Proposition \ref{prop:integral_influence}. 
\end{cor}

\begin{proof}[Proof of Corollary \ref{cor:integral_influence}] For any bounded function $g \in \dot{\mathcal P}$, $P_t=(1+tg)P$ and any fixed $z\in \R^k$,
\begin{align*} 
\Psi(P_t) 
& = \Psi(P) + t \int \widetilde \psi^P(y) g(y) d P(y) + t \int \psi_z(P) g(z) dP(z) + o(t) 
\end{align*}
using Assumption \eqref{eq:condi}.
Then
$$  \lim_{t \to 0^+} \frac{\Psi(P_t) - \Psi(P)} t = \int \big( \widetilde \psi^P(y) + \psi_y(P) \big) g(y) dP (y),$$
and $y \mapsto \widetilde \psi^P(y) + \psi_y(P) $ is an influence function. As $\widetilde \psi^P$ lies in $\dot{\mathcal P}$, the result follows. 
\end{proof}


\section{Pick-Freeze setting}\label{sec:PF}


In the Pick-Freeze sampling scheme, we observe $n$ independent draws from a pair of responses $(Y, Y')$ i.i.d.\ conditionally to the input $X$. Typically, if $Y$ can be expressed as $Y = G(X, W)$ for some function $G$, with $W$ a random term independent from $X$, the copy $Y'$ is obtained as $Y' = G(X, W')$ where $W$ and $W'$ are i.i.d. In this context, it is convenient to define the model $\mathcal P$ as the set of possible distributions of $(Y, Y')$, thus omitting the dependency in the input $X$, since it does not appear in the calculation of the efficiency bound. 

\medskip

Let $F_{Y,Y'}$ denotes c.d.f.\ of $P$, that is, for $(Y,Y') \sim P$ :
$$ \forall y, y' \in \mathbb R^k \, , \, F_{Y,Y'}(y, y') = \mathbb P (Y \leq y, Y' \leq y' ). $$
Note that the exchangeability of $(Y, Y')$ is equivalent to the symmetry of the c.d.f.\ $F_{Y,Y'}$ in its two arguments:
\[
F_{Y,Y'}(y,y')=F_{Y,Y'}(y',y) \;\;\; \forall (y,y')\in\R^2. 
\]

Moreover, we recall that $F_Y$ denotes 
 the marginal c.d.f.\ of $P$. 
Similarly, let $F_Z$ denote the c.d.f.\ of $Z$ drawn from $Q$ and $F_Z^-$ its left-side limit:
$$ \forall z \in \mathbb R^k \, , \, F_Z^-(z) = \mathbb P(Z < z). $$

\subsection{Sobol' indices 
}\label{ssec:PF_sobol}

Here the model $\mathcal P$ is the set of exchangeable distributions on $\mathbb R^k \times \mathbb R^k$, defined as 
$$\mathcal{P}:=\{ P\in\mathcal{P}(\mathbb R^k \times \mathbb R^k ): \forall A , B \in \mathcal B(\mathbb R^k) \, , \, P(A \times B) = P(B \times A) \} .  $$

The exchangeability condition entails that a differentiable submodel $(P_t)_{t \geq 0}$ has a symmetric score function $g$, as pointed out in \cite{klein2024note_efficiency}. Recall that $P$ stands for the distribution of $(Y,Y')$. We check that $P \in \mathcal P$ thanks to \cite[Lemma 2.4]{janon2012asymptotic}. The tangent set $\dot{\mathcal P}$ of $\mathcal P$ at $P$ is then given by 
\begin{equation}\label{eq:tangent1}
    \dot{\mathcal P} = \big\{ g \in L^2(P) :
  \E [g(Y, Y')]  = 0  \text{ and }   g(y, y') = g(y', y) \, , \,  \forall y, y' \in \R^k \big\}.
\end{equation}

\begin{prop}\label{prop:EIF_PF_sobol} \cite{janon2012asymptotic}
In the Pick-Freeze setting, if $\E[(YY')^2]<\infty$, the parameters 
\begin{itemize}
    \item $\psi(P) = \E[YY']$ has efficient influence function
    $ \widetilde \psi (y,y') = y y' -\psi(P),  $
    \item $\phi_1(P) = \E[Y^2]$ has efficient influence function
    $\widetilde \phi_1 (y,y') = \frac12 (y^2 + y'{^2})-\phi_1(P). $
    \item $\phi_2(P) = \E[Y]^2$ has efficient influence function
    $\widetilde \phi_2 (y,y') = \mathbb E[Y] (y + y')  - 2 \phi_2(P).$
\end{itemize}
\end{prop}

\begin{proof}[Proof of Proposition  \ref{prop:EIF_PF_sobol}]
For any bounded function $g \in \dot{\mathcal P}$ and $P_t=(1+tg)P$,  we have
\begin{eqnarray*}
	\frac{\psi(P_t)-\psi(P)}{t} &=& \E [YY' g(Y,Y') ] =\E \left[ ( YY'-\E[YY']) g(Y,Y') \right].
\end{eqnarray*}
Since the map $(y,y') \mapsto  y y'-\E[YY']   $ lies in $\dot{\mathcal P}$ (i.e.~it is symmetrical and $P$-square-integrable by assumption), it is the efficient influence function for the parameter $\E[YY']$. Proceeding in the same way, we show that $(y,y') \mapsto y - \mathbb E[Y]$ is an influence function for the parameter $\E[Y]$, as is $(y,y') \mapsto y' - \mathbb E[Y]$ by exchangeability. We can thus identify the efficient influence function 
$$ (y,y') \mapsto \frac12 (y + y')  - \mathbb E[Y] $$
as the only symmetrical influence function of $\mathbb E[Y]$. Finally, by the chain rule, 
$ \E[Y] (y + y')  - 2\mathbb E[Y]^2 $ 
is the efficient influence function of $\phi_2(P)$. 
The reasoning is the same for  $\E[Y^2]$.
\end{proof}

\subsection{Generalized Cramér-von Mises indices 
- case \texorpdfstring{$Q$}{Q} known}\label{ssec:PF_CVM_mu_known}

In this section, we consider the estimation of the generalized Cramér-von Mises  index  recalled in \eqref{eq:formule} and \eqref{def:CVMmu}. If $Q$ is known, the model $\mathcal P$ involves only the distribution $P$ of $(Y, Y')$, which is known to be exchangeable. Thus, as in Section \ref{ssec:PF_sobol}, the model $\mathcal P$ is the set of exchangeable distributions on $\mathbb R^k \times \mathbb R^k$ and the tangent set $\dot{\mathcal P}$ of $\mathcal P$ at $P$ is  given by \eqref{eq:tangent1}.

\begin{prop}
\label{prop:EIF_PF_mu_known}
In the Pick-Freeze setting where $Q$ is known, the parameters 
\begin{itemize}
    \item $\psi^Q(P) = \displaystyle \int_{\mathbb R^k} \P(Y\leqslant z, Y'\leqslant z) \, d Q(z) $ has efficient influence function
    $$ \widetilde \psi^Q (y,y') = 1-F_Z^-(y\vee y') -\psi^Q(P),  $$
    \item $\phi^Q_1(P) = \displaystyle \int_{\R^k} \P(Y\leqslant z)dQ(z) $ has efficient influence function
    $$ \widetilde \phi^Q_1 (y,y') =  1 - \frac 12 (F_Z^-(y)+F_Z^-(y'))  - \phi^Q_1(P), $$
    \item $\phi^Q_2(P) = \displaystyle \int_{\R^k} \P(Y\leqslant z)^2dQ(z)$ has efficient influence function
     $$ \widetilde \phi^Q_2 (y,y') = \int_{\R^k}  \P(Y\leqslant z ) (\ind_{\{y\leqslant z\}}+\ind_{\{y'\leqslant z\}}) dQ(z) -2\phi^Q_2(P).  $$
\end{itemize}
\end{prop}

\begin{proof}[Proof of Proposition  \ref{prop:EIF_PF_mu_known}]

For $z \in \mathbb R^k$, consider the two parameters on $\mathcal P$ given by
\begin{align*}
\psi_z(P) & 
=\P(Y \leqslant z, Y'\leqslant z) 
\quad \text{and} \quad \phi_z(P)  
=\P (Y\leqslant z),
\end{align*} 
whence 
\begin{align*}
\psi^Q(P) & =\int \psi_z(P) dQ(z),\quad 
\phi_1^Q(P)  = \int \phi_z(P) dQ(z), \quad \text{and} \quad 
\phi_2^Q(P)  = \int \phi_z^2(P) dQ(z), 
\end{align*} 
and it remains to apply Proposition \ref{prop:integral_influence}. For any bounded function $g \in \dot{\mathcal P}$, $P_t=(1+tg)P$ and any fixed $z\in \R^k$, we have
\begin{align*}
	\frac{\psi_z(P_t) - \psi_z(P)} t &= \E [\ind_{\{Y\leqslant z,\, Y'\leqslant z\}} g(Y,Y') ] 
	   =\E \left[ ( \ind_{\{Y\leqslant z,\, Y'\leqslant z\}}-\psi_z(P)) g(Y,Y') \right],
\end{align*}
from which we deduce that the parameter $\psi_z(P)$ satisfies assumption \eqref{eq:condi}. 
Since the map 
\[
\widetilde \psi_{z}\colon (y,y') \mapsto  \ind_{\{y\leqslant z,\, y'\leqslant z\}}-\psi_z(P)
\]
lies in $\dot{\mathcal P}$, it is the efficient influence function at $P$ of the parameter $\psi_z(P)$ for any fixed $z\in \R^k$. 

\medskip 

Proceeding in the same way, we show that, for any $z\in \R^k$, $(y,y') \mapsto \ind_{\{y\leqslant z\}} - \phi_z (P)$  is an influence function for the parameter $\phi_z (P)= \P(Y\leqslant z)$, as is $(y,y') \mapsto \ind_{\{y'\leqslant z\}} - \phi_z (P)$ by exchangeability. Clearly, the parameter $\phi_z(P)$ satisfies assumption \eqref{eq:condi}.
We can thus identify the efficient influence function 
$$\widetilde \phi_z \colon (y,y') \mapsto \frac 12 (\ind_{\{y\leqslant z\}}+\ind_{\{y'\leqslant z\}})  -  \phi_z (P)$$
as the only symmetrical influence function of $\phi_z(P)$.

\medskip

Finally, by the chain rule, 
$2 \phi_z(P) \widetilde \phi_z$ 
is the efficient influence function of the parameter $P \mapsto \phi_z(P)^2$. The result follows from the item 1 of Proposition \ref{prop:integral_influence}.
\end{proof}


\subsection{Generalized Cramér-von Mises  indices 
- case \texorpdfstring{$Q$}{Q} unknown}\label{ssec:PF_CVM_mu_unknown}

In this section, we consider the generalized Cramér-von Mises  index  recalled in \eqref{eq:formule} and \eqref{def:CVMmu} when $Q$ is unknown. Hence, the model $\mathcal Q$ on $Q$ is the non-parametric model on $\mathbb R^k$ while the model $\mathcal P$ still designates the set of exchangeable distributions on $\mathbb R^k \times \mathbb R^k$. Let $P \!\otimes\! Q$ be the distribution of $(Y,Y',Z)$. By Lemma \ref{lem:tangent_prod}, the tangent set of $\mathcal P \!\otimes\! \mathcal Q$ at $P \!\otimes\! Q$ is the direct sum
$$ \dot{\mathcal P \!\otimes\! \mathcal Q} = \dot{\mathcal P} \oplus \dot{\mathcal Q}  $$
where $\dot{\mathcal P}$ as defined in \eqref{eq:tangent1} while $\dot{\mathcal Q}$ is the maximal tangent set 
$$ \dot{\mathcal Q} = \Big\{ h \in L^2(Q) :\, \int h \, d Q = 0 \Big\}.   $$


\begin{prop}
\label{prop:EIF_PF_mu_unknown}
In the Pick-Freeze setting where $Q$ is unknown, the parameters 
\begin{itemize}
    \item $\Psi(P \! \otimes \! Q) = \psi^Q(P) = \displaystyle \int_{\mathbb R^k} \P(Y\leqslant z, Y'\leqslant z) \, d Q(z)$ has efficient influence function
    $$ \widetilde \Psi (y,y',z) = 1 - F_Z^-(y \vee y') + F_{Y,Y'}(z,z) - 2 \psi^Q(P),  $$
    \item $\Phi_1 (P\! \otimes \! Q) = \phi^Q_1(P) = \displaystyle \int_{\mathbb R^k} \mathbb P(Y \leq z) \, d Q(z)$ has efficient influence function
    $$ \widetilde \Phi_1 (y,y',z) =   1 - \frac 12 (F_Z^-(y)  + F_Z^-(y')) + F_{Y}(z) - 2 \phi^Q_1(P), $$
    \item $\Phi_2(P\! \otimes \! Q) = \phi_2^Q(P) = \displaystyle \int_{\mathbb R^k} \mathbb P(Y \leq z)^2 \, d Q(z)$ has efficient influence function
 $$ \widetilde \Phi_2 (y,y',z) = \int_{\R^k} \P(Y\leqslant z)(\ind_{\{y\leqslant z\}}+\ind_{\{y'\leqslant z\}})]dQ(z)+ \P(Y\leqslant z)^2 -3 \phi^Q_2(P).  $$
\end{itemize}

\end{prop}

\begin{proof}[Proof of Proposition  \ref{prop:EIF_PF_mu_unknown}]
It suffices to apply item 2 of Proposition \ref{prop:integral_influence} using the expressions of $\widetilde \psi_z$ and $\widetilde \phi_z$ computed in the proof of Proposition \ref{prop:EIF_PF_mu_known}. 
\end{proof}

\subsection{Cramér-von Mises indices 
} \label{ssec:PF_CVM_mu_Y}

In this section, we consider the estimation of the Cramér-von Mises indices in \eqref{def:CVM}  when $Q$ is the (shared) marginal distribution of $Y$ and $Y'$, denoted by $P_Y$. Recall that $P$ stands for the distribution of $(Y, Y')$.
Consider a $4$-uplet $(Y, Y', Z, Z')$ with distribution $P \!\otimes\! P$.

\begin{prop}
\label{prop:EIF_PF_mu_Y}
In the Pick-Freeze setting where $Q=P_Y$, the parameters
\begin{itemize}
    \item $\Psi(P) = \psi^P(P) =  \displaystyle \int_{\mathbb R^k} \P(Y\leqslant z, Y'\leqslant z)\, d P_Y(z)$ has efficient influence function
    $$ \widetilde \Psi (y,y')  = 1 - F_Y^-(y \vee y') +\frac12 ( F_{Y,Y'}(y, y) + F_{Y,Y'}(y', y') ) 
 - 2\psi^P(P),  $$
    \item $\Phi_1 (P) = \phi^P_1(P) = \displaystyle \int_{\mathbb R^k} \mathbb P(Y \leq z) \, d P_Y(z)$ has efficient influence function
    $$ \widetilde \Phi_1 (y,y') =  1 +\frac 12(  \P(Y=y) + \P(Y=y')) -  2\phi^P_1(P), $$
    \item $\Phi_2(P) = \phi_2^P(P) = \displaystyle \int_{\mathbb R^k} \mathbb P(Y \leq z)^2 \, d P_Y(z)$ has efficient influence function
        $$ \widetilde \Phi_2 (y,y') = \int_{\R^k} \P(Y\leqslant z)(\ind_{\{y\leqslant z\}}+\ind_{\{y'\leqslant z\}})\, d P_Y(z) + \frac 12 (F_Y^2(y)+F_Y^2(y'))- 3\phi^P_2(P).  $$
\end{itemize}
\end{prop}

\begin{proof}[Proof of Proposition \ref{prop:EIF_PF_mu_Y}]
It suffices to apply Corollary \ref{cor:integral_influence} and Proposition \ref{prop:EIF_PF_mu_unknown}. 
\end{proof}

\subsection{Extension to universal sensitivity indices}
 \label{ssec:PF_univ}

It suffices to slightly adapt the proofs of Propositions \ref{prop:EIF_PF_mu_known} and \ref{prop:EIF_PF_mu_unknown} to get the same results for the universal indices defined in \eqref{def:univ}, provided that all test functions $T_z(Y)$ at $Y$ are square-integrable and the associated parameters $\psi_z = \mathbb E [ T_z(Y)T_z(Y') ]$, $\phi_{1,z} = \mathbb E[ T_z(Y)^2]$, and $\phi_{2,z} = \mathbb E[ T_z(Y)]^2$ all satisfy Condition \eqref{eq:condi} of Proposition \ref{prop:integral_influence}.

\medskip

When $Q$ is known, the parameters 
\begin{itemize}
    \item $\psi^Q(P) = \displaystyle \int_{\mathcal Z} \E[T_z(Y)T_z(Y')]\, d Q(z) $ has efficient influence function
    $$ \widetilde \psi^Q (y,y') = \int_{\mathcal Z} \E[T_z(y)T_z(y')]\, dQ(z) -\psi^Q(P),  $$
    \item $\phi^Q_1(P) = \displaystyle \int_{\mathcal Z} \E[T_z(Y)^2]\, dQ(z) $ has efficient influence function
    $$ \widetilde \phi^Q_1 (y,y') =  \frac 12 \int_{\mathcal Z} (T_z(y)^2+T_z(y')^2)\, dQ(z) - \phi^Q_1(P), $$
    \item $\phi^Q_2(P) = \displaystyle \int_{\mathcal Z} \E[T_z(Y)]^2\, dQ(z)$ has efficient influence function
     $$ \widetilde \phi^Q_2 (y,y') =  \int_{\mathcal Z} \E[T_z(Y)] (T_z(y)+T_z(y'))\, dQ(z)  -2\phi^Q_2(P).  $$
\end{itemize}

\medskip

When $Q$ is unknown, the parameters 
\begin{itemize}
    \item $\Psi(P \! \otimes \! Q) $ has efficient influence function
    $\widetilde \Psi (y,y',z) = \widetilde \psi (y,y') + \E[T_z(Y)T_z(Y')] -  \psi^Q(P), $
    \item $\Phi_1 (P\! \otimes \! Q) $ has efficient influence function
    $ \widetilde \Phi_1 (y,y',z) =    \widetilde \phi_1 (y,y')+ \E[T_z(Y)^2] -  \phi^Q_1(P), $
    \item $\Phi_2(P\! \otimes \! Q) $ has efficient influence function
 $ \widetilde \Phi_2 (y,y',z) =  \widetilde \phi_2 (y,y') + \E[T_z(Y)]^2 - \phi^Q_2(P).  $
\end{itemize}

\section{Given-Data setting}\label{sec:given_data} 

In this section, we consider the estimation of the Sobol' index, the  Cramér-von Mises index and its generalized version with respect to $X$ 
 using a Given-Data design of experiment. In the rest of this section, $P$ stands for the distribution of $(X,Y)$.

\subsection{Sobol' indices 
}\label{ssec:given_data_sobol}

In this section, we observe a sample $(X_1,Y_1),\dots,(X_n,Y_n)$ drawn independently from the distribution $P$. The model $\mathcal{P}$ contains all the distributions on $\R^d \times \R$. As pointed out in \cite[Remark 25.16]{van2000asymptotic}, in this model, the tangent set $\dot{\mathcal P}$ at $P \in \mathcal P$ is the maximal tangent set, containing all $P$-square-integrable functions $g$ with zero integral
\begin{equation}\label{eq:tangent2}
\dot{\mathcal P} = \{ g \in L^2(P) : \E[ g(X,Y)] = 0 \}. 
\end{equation}

Dealing with the parameter  $\psi(P) = \E[\E[Y|X]^2]$, the efficient influence function has been given in \cite{doksum1995nonparametric} (without proof) and in \cite{da2008efficient} under absolute continuity conditions.  For completeness, we here calculate the efficient influence function in the general case.

\begin{prop}\label{prop:EIF_given_data_sobol}
In the Given-Data setting, if $\E[Y^4]<\infty$, the parameters 
\begin{itemize}
    \item $\psi(P) = \E[\E[Y|X]^2]$ has efficient influence function
     $$\widetilde \psi (x,y)= (2y-m(x) )m(x) - \psi(P) \quad \text{where $m(x) = \mathbb E[Y | X = x ]$},$$ 
    \item $\phi_1(P) = \E[Y^2]$ has efficient influence function
    $\widetilde \phi_2 (x,y) = y^2-\phi_1(P), $
    \item $\phi_2(P) = \E[Y]^2$ has efficient influence function
    $\widetilde \phi_2 (x,y) = 2(\E[Y] y - \phi_2(P)). $
\end{itemize}
\end{prop}

\begin{proof}[Proof of Proposition  \ref{prop:EIF_given_data_sobol}]
Let $P_t = (1+tg) P$ with $g$ bounded, $(X_t,Y_t)$ with distribution $P_t$ and denote  the conditional expectation function under $P_t$ by
$$ m_t : x\in \R^d \mapsto \mathbb E [  Y_t | X_t = x ]   $$
that satisfies
\begin{equation}\label{eq:infl_ident} \mathbb E [  Y_t h(X_t) ]  = \mathbb E [ m_t(X_t) h(X_t) ]  \end{equation}
for all measurable function $h : \mathbb R^p \to \mathbb R$ such that $\mathbb E [|  Y_t h(X_t) |] < + \infty $. 
From
\[
\psi(P_t)= \mathbb E[ Y_t m_t(X_t) ] = \mathbb E [ Y m_t(X) ] + t \, \mathbb E [Y m_t(X) g(X,Y) ], 
\]
we deduce
\begin{align*} \frac{\psi(P_t) - \psi(P)} t & =  \mathbb E \Big[ \frac{m_t(X) - m(X)} t  Y \Big] +  \mathbb E \big[  Y m_t(X) g(X,Y) \big]
\end{align*}
recalling that $\E[Ym(X)]=\E[m^2(X)]$. Taking $h=m$ in  \eqref{eq:infl_ident} yields in particular
$$  \mathbb E \big[  Y m(X) \big(1 +t g(X, Y) \big)\big] =  \mathbb E \big[ m_t(X) m(X) \big(1 +t g(X, Y) \big) \big]      $$
leading to
$$ \mathbb E \Big[ \frac{m_t(X) - m(X)} t  m(X) \Big] =  \mathbb E \big[ (Y - m_t(X) ) m(X) g(X, Y) \big]  $$
whence
\begin{align*} \frac{\psi(P_t) - \psi(P)} t 
& = \mathbb E \big[ \big( (Y - m_t(X) ) m(X) + Y m_t(X) \big) g(X,Y) \big]. 
\end{align*}
We get
\begin{align*}  
\Big| \frac{\psi(P_t) - \psi(P)} t  &- \mathbb E \big[  \big( 2Y - m(X) \big)  m(X) g(X,Y) \big] \Big| \\
& =  \Big| \mathbb E \big[ (m_t(X) - m(X)) (Y - m(X)) g(X, Y) \big] \Big| \\
& \leq \Vert g \Vert_\infty \mathbb E [(Y - m(X))^2 ] ^{1/2} \mathbb E [(m_t(X) - m(X))^2 ] ^{1/2}\\
& \leq 2 \Vert g \Vert_\infty \mathbb E [Y^2 ] ^{1/2} \mathbb E [(m_t(X) - m(X))^2 ] ^{1/2}
\end{align*}
by Jensen's inequality. Now, by direct calculation, we verify that
\begin{align*}
\mathbb E \big[  \big( m_t( X) - m( X) \big)^2 \big] & = t \, \mathbb E \big[  \big( m_t( X) - m( X) \big) \big(  Y -m(X) \big) g(X,Y) \big]  \\
& \leqslant t \, \Vert g \Vert_\infty \mathbb E \big[  \big( m_t( X) - m( X) \big) ^2\big]^{1/2} \mathbb E \big[  \big(  Y -m(X) \big) ^2 \big]^{1/2}\\
& \leqslant 2 t \, \mathbb E [Y^2 ] ^{1/2} \Vert g \Vert_\infty \mathbb E \big[  \big( m_t( X) - m( X) \big) ^2\big]^{1/2} 
\end{align*}
by Cauchy-Schwarz and Jensen's inequalities,  whence 
\begin{align}\label{eq:condi_verif}
   \E \big[ \big( m_t( X) - m( X) \big)^2 \big] \leq 4 t^2 \,  \Vert g \Vert_\infty^2 \mathbb E [Y^2 ] .
\end{align}
Hence,
$$ \lim_{t \to 0^+} \frac{\psi(P_t) - \psi(P)} t  =  \mathbb E \big[  \big( (2Y - m(X) ) m(X) - \psi(P) \big) g(X,Y) \big] $$
revealing $(x,y) \mapsto (2 y   - m(x) )m(x)- \psi(P)
\in \dot{\mathcal P}$ as the efficient influence function. 

\medskip

Because both parameters $\phi_1(P)$  and $\phi_2(P)$ are linear, they are particularly easy to deal with. For instance, 
$$ \frac{\phi_1(P_t) - \phi_1(P)}{t} = \mathbb E [Y g(X,Y) ] =  \mathbb E \big[(Y - \mathbb E[Y]) g(X,Y) \big] . $$
We conclude easily that the efficient influence functions of $\phi_1(P)$  and $\phi_2(P)$ are
\[
(x,y)\mapsto y^2-\phi_1(P) \quad \text{and} \quad (x,y)\mapsto 2(\E[Y] y-\phi_2(P)),
\]
although the first requires the additional condition $\mathbb E[Y^4] < \infty$  to lie in the tangent set $\dot{\mathcal P}$. 
\end{proof}

\subsection{Generalized Cramér-von Mises indices 
- case \texorpdfstring{$Q$}{Q} known}\label{ssec:given_data_CVM_mu_known}

In this section, we consider the estimation of the generalized Cramér-von Mises index recalled in \eqref{eq:formule} and \eqref{def:CVMmu} when $Q$ is known, in the Given-Data setting where we observe a sample $(X_1,Y_1),\dots,(X_n,Y_n)$ drawn independently from the distribution $P$. Then, as in Section \ref{ssec:given_data_sobol}, the model $\mathcal{P}$ contains all the distributions on $\R^d \times \R$ and the tangent set $\dot{\mathcal P}$ is given in \eqref{eq:tangent2}.

\medskip

For all $z \in \mathbb R^k$, $p(z | X) = \P(Y \leq z | X)$ denotes the conditional c.d.f.\ of $Y$ knowing $X$, characterized by the equality
$$ \mathbb E \big[  p(z | X) h(X) \big] = \mathbb E \big[ \ind_{\{ Y \leq z \}} h(X) \big]  $$
which holds for all integrable function $h: \mathbb R^d \to \mathbb R$.

\begin{prop}
\label{prop:EIF_given_data_mu_known}

In the Given-Data setting where $Q$ is known, the parameters 
\begin{itemize}
    \item $\psi^Q(P) = \displaystyle \int_{\mathbb R^k} \mathbb E \big[ p(z | X)^2 \big]\, d Q(z)$ has efficient influence function
    $$ \widetilde \psi^Q (x,y) = \int_{\R^k} (2 \ind_{\{y\leqslant z\}}   - p(z | x) )p(z | x) \, dQ(z) -\psi^Q(P),  $$
    \item $\phi^Q_1(P) = \displaystyle \int_{\mathbb R^k} \mathbb P(Y \leq z) \, d Q(z)$ has efficient influence function
    $$ \widetilde \phi^Q_1 (x,y) = 
    1-F_Z^-(y)- \phi^Q_1(P), $$
    \item $\phi^Q_2(P) = \displaystyle \int_{\mathbb R^k} \mathbb P(Y \leq z)^2 \, d Q(z)$ has efficient influence function
    $$ \widetilde \phi^Q_2 (x,y) = 2 \int_{\mathbb R^k}  \mathbb P(Y \leq z) \ind_{\{y\leqslant z\}}d Q(z)  -2\phi^Q_2(P).  $$
\end{itemize}
\end{prop}

\begin{proof}[Proof of Proposition  \ref{prop:EIF_given_data_mu_known}]
Let $P_t=(1+tg)P$ with $g$ bounded, $(X_t,Y_t)$ with distribution $P_t$ and denote by 
\begin{equation}\label{eq:pt} p_t(z | x) = \mathbb E[ \ind_{ \{ Y_t \leq z \}} | X_t = x ]
\end{equation}
for $x\in \R^d$ and  $z\in \R^k$ fixed.  Recall that, by definition of the conditional expectation, 
\begin{equation}\label{eq:infl_ident2} \mathbb E [ \ind_{ \{ Y_t \leq z \}} h(X_t) ]  = \mathbb E [ p_t(z | X_t) h(X_t) ]  \end{equation}
holds for all bounded function $h : \mathbb R^d \to \mathbb R$. Noting that $\psi_z(P_t)=\E[p_t(z | X)^2]=\E[\ind_{ \{ Y_t \leq z \}} p_t(z | X)]$, we get
\[
\psi_z(P_t) =   \mathbb E [ \ind_{ \{ Y_t \leq z \}} p_t(z | X_t) ] = \mathbb E [ \ind_{ \{ Y \leq z \}} p_t(z | X) ] + t \, \mathbb E [\ind_{ \{ Y \leq z \}} p_t(z | X) g(X,Y) ], 
\]
leading to 
\begin{align*} \frac{\psi_z(P_t) - \psi_z(P)} t & =  \mathbb E \Big[ \frac{p_t(z | X) - p(z | X)} t  \ind_{ \{ Y \leq z \}} \Big] +  \mathbb E \big[ \ind_{ \{ Y \leq z \}} p_t(z | X) g(X,Y) \big].
\end{align*}
 Taking $h = p(z \, | \, .)$ in  \eqref{eq:infl_ident2} yields in particular
$$  \mathbb E \big[ \ind_{ \{ Y \leq z \}} p(z | X) \big(1 +t g(X, Y) \big)\big] =  \mathbb E \big[ p_t(z | X) p(z | X) \big(1 +t g(X, Y) \big) \big]      $$
leading to
$$ \mathbb E \Big[ \frac{p_t(z | X) - p(z | X)} t  p(z | X) \Big] =  \mathbb E \big[ (\ind_{ \{ Y \leq z \}} - p_t(z | X) ) p(z | X) g(X, Y) \big]  $$
whence
\begin{align*} \frac{\psi_z(P_t) - \psi_z(P)} t 
& = \mathbb E \big[ \big( (\ind_{ \{ Y \leq z \}}  - p_t(z | X) ) p(z | X) + \ind_{ \{ Y \leq z\}} p_t(z | X) \big) g(X,Y) \big]. 
\end{align*}
We get
\begin{align*}  
\Big| \frac{\psi_z(P_t) - \psi_z(P)} t  &- \mathbb E \big[  \big( 2\ind_{\{Y\leqslant z\}} - p(z | X) \big)  p(z | X) g(X,Y) \big] \Big| \\
& =  \Big| \mathbb E \big[ (p_t(z | X) - p(z | X)) (\ind_{\{ Y\leqslant z\}} - p(z | X)) g(X, Y) \big] \Big| \\
& \leq  \Vert g \Vert_\infty \mathbb E \big[  \big( p_t(z | X) - p(z | X) \big)^2 \big]^{1/2} \mathbb E \big[  \big( \ind_{\{ Y\leqslant z\}} - p(z | X) \big)^2 \big]^{1/2} \\
& \leq 2 \Vert g \Vert_\infty \mathbb E \big[  \big( p_t(z | X) - p(z | X) \big)^2 \big]^{1/2}. 
\end{align*}
by Cauchy-Schwarz inequality. Using the same arguments as in Equation \eqref{eq:condi_verif}, 
we get \begin{align}\label{eq:condi_verif2}
    \sup_{z\in \R^k} \E \big[ \big( p_t(z | X) - p(z | X) \big)^2 \Big] \leq 4 t^2 \, \Vert g \Vert_\infty.
\end{align}
Then we deduce that 
$$ \lim_{t \to 0^+} \frac{\psi_z(P_t) - \psi_z(P)} t  =  \mathbb E \big[   (2\ind_{\{Y\leqslant z\}} - p(z | X) ) p(z | X) g(X,Y) \big] $$
revealing $\widetilde \psi_{z}\colon (x,y) \mapsto (2 \ind_{\{y\leqslant z\}}   - p(z | x) )p(z | x)- \psi_z(P)
\in \dot{\mathcal P}$ as the efficient influence function of the parameter $\psi_z$. We recover the wanted result by item 1 of Proposition \ref{prop:integral_influence} (Assumption \eqref{eq:condi} holds in view of \eqref{eq:condi_verif2}). 

As for $\phi_1^Q(P)$ and $\phi_2^Q(P)$, we start by checking that $\widetilde \phi_z(y)=\ind_{y\leqslant z}-\P(Y\leqslant z)$ is the efficient influence function of the parameter $P \mapsto \phi_z(P)=\P(Y\leqslant z)$. It remains to apply item 1 of Proposition \ref{prop:integral_influence} to get the result for $\phi_1^Q(P)$. Finally, by the chain rule, 
$2 \phi_z(P) \widetilde \phi_z$ 
is the efficient influence function of the parameter $P \mapsto \phi_z(P)^2$. The result for $\phi_2^Q(P)$ follows from item 1 of Proposition \ref{prop:integral_influence}.
\end{proof}

\subsection{Generalized Cramér-von Mises  indices 
- case \texorpdfstring{$Q$}{Q} unknown}\label{ssec:given_data_CVM_mu_unknown}

In this section, we consider the estimation of the generalized Cramér-von Mises index recalled in \eqref{eq:formule} and \eqref{def:CVMmu} when $Q$ is unknown. We now consider the Given-Data setting where we observe a sample \[
(X_1,Y_1,Z_1),\dots,(X_n,Y_n,Z_n)
\]
drawn independently from some distribution $ P\!\otimes\! Q$. Hence, the model is 
$$\mathcal{P} \! \otimes \! \mathcal Q = \{ P \! \otimes \! Q : P \in \mathcal P , Q \in \mathcal Q \}$$
and the tangent set at $P\!\otimes\! Q$ is the direct sum $\dot{\mathcal P} \!\oplus\! \dot{\mathcal Q}$ by Lemma \ref{lem:tangent_prod}. 



\begin{prop}
\label{prop:EIF_given_data_mu_unknown}
In the Given-Data setting where $Q$ is unknown, the parameters 
\begin{itemize}
    \item $\Psi(P \! \otimes \! Q) = \psi^Q(P) = \displaystyle \int_{\mathbb R^k} \mathbb E \big[ p(z | X)^2 \big]\, d Q(z)$ has efficient influence function
    $$ \widetilde \Psi  (x,y,z)  = \int_{\R^k} (2 \ind_{\{y\leqslant z'\}}   - p(z' | x) )p(z' | x) \, dQ(z') + \E[p(z|X)^2] - 2 \psi^Q(P),  $$
    \item $\Phi_1 (P\! \otimes \! Q) = \phi^Q_1(P) = \displaystyle \int_{\mathbb R^k} \mathbb P(Y \leq z) \, d Q(z)$ has efficient influence function
    $$ \widetilde \Phi_1 (x,y,z) 
    =1-F_Z^-(y) + F_Y( z)  - 2 \phi^Q_1(P), $$
    \item $\Phi_2(P\! \otimes \! Q) = \phi_2^Q(P) = \displaystyle \int_{\mathbb R^k} \mathbb P(Y \leq z)^2 \, d Q(z)$ has efficient influence function
    $$ \widetilde \Phi_2 (x,y,z) = 2 \int_{\R^k}  F_Y( z') \ind_{\{y\leqslant z'\}}\, d Q(z') +  F_Y(z)^2 -3 \phi^Q_2(P).  $$
\end{itemize}
\end{prop}

\begin{proof}[Proof of Proposition  \ref{prop:EIF_given_data_mu_unknown}]
It suffices to apply item 2 of Proposition \ref{prop:integral_influence} using the expressions of $\widetilde \psi_z$ and $\widetilde \phi_z$ computed in the proof of Proposition \ref{prop:EIF_given_data_mu_known}.
\end{proof}

\subsection{Cramér-von Mises indices 
} 
\label{ssec:given_data_CVM_mu_Y}

In this section, we consider the estimation of the Cramér-von Mises recalled  in \eqref{def:CVM} with $Q=P_Y$, where $P_Y$ is the marginal distribution of $Y$.
Let $(X,Y, Z)$ with distribution $P \!\otimes\! P_Y$. 

\begin{prop}
\label{prop:EIF_given_data_mu_Y}
In the Given-Data setting where $Q=P_Y$, the parameters 
\begin{itemize}
    \item $\Psi(P) = \psi^P(P) = \displaystyle \int_{\mathbb R^k} \mathbb E \big[ p(z | X)^2 \big]\, d P_Y(z)$ has efficient influence function
    $$ \widetilde \Psi  (x,y)  = \E \big[ (2 \ind_{\{y\leqslant Y\}}   - p(Y| x) )p(Y | x) \big] + \E[p(y|X)^2] - 2\psi^P(P),  $$
    \item $\Phi_1 (P) = \phi^P_1(P) = \displaystyle \int_{\mathbb R^k} \mathbb P(Y \leq z) \, d P_Y(z)$ has efficient influence function
    $$ \widetilde \Phi_1 (x,y) =  1- \mathbb P( Y = y) -2\phi^P_1(P), $$
    \item $\Phi_2(P) = \phi_2^P(P) = \displaystyle \int_{\mathbb R^k} \mathbb P(Y \leq z)^2 \, d P_Y(z)$ has efficient influence function
    $$ \widetilde \Phi_2 (x,y) = 2 \E[  F_Y(Y) \ind_{\{y\leqslant Y\}}] +F_Y(y)^2 -3\phi^P_2(P).  $$
\end{itemize}
\end{prop}

\begin{proof}[Proof of Proposition \ref{prop:EIF_given_data_mu_Y}]
It suffices to apply Corollary \ref{cor:integral_influence} and Proposition \ref{prop:EIF_given_data_mu_unknown}. 
\end{proof}

\begin{rem} For a real-valued $Y$, the Cram\'er-von Mises index, obtained by integrating over the distribution $P_Y$, 
has constant $\Phi_1(P) = 1/2$ and $\Phi_2(P) = 1/3$ for a marginal distribution $P_Y$ with continuous c.d.f.\ (i.e.~with no atom), and the efficient influence function is trivially zero in this case. 
\end{rem}

\subsection{Extension to universal sensitivity indices}
 \label{ssec:gicen_data_univ}

It suffices to slightly adapt the proofs of Propositions \ref{prop:EIF_given_data_mu_known} and \ref{prop:EIF_given_data_mu_unknown} to get the same results for the universal indices defined in \eqref{def:univ}, provided that all test functions $T_z(Y)$ at $Y$ are square-integrable and the associated parameters $\psi_z = \mathbb E [ \mathbb E[ T_z(Y) | X]^2 ]$, $\phi_{1,z} = \mathbb E[ T_z(Y)^2]$, and $\phi_{2,z} = \mathbb E[ T_z(Y)]^2$ all satisfy Condition \eqref{eq:condi} of Proposition \ref{prop:integral_influence}.

\medskip

For all $z \in \mathcal Z$, we replace the conditional c.d.f.\ $p(z| X)=\P(Y\leqslant z| X)$ by the conditional expectation $q(z | X) = \E[T_z(Y) | X]$ of the test function $T_z$ knowing $X$. 

\medskip

When $Q$ is known, the parameters 
\begin{itemize}
    \item $\psi^Q(P) = \displaystyle \int_{\mathcal Z} \mathbb E \big[ q(z | X)^2 \big]\, d Q(z)$ has efficient influence function
    $$ \widetilde \psi^Q (x,y) = \int_{\mathcal Z} (2 T_z(y)  - q(z | x) )q(z | x) \, dQ(z) -\psi^Q(P),  $$
    \item $\phi^Q_1(P) = \displaystyle \int_{\mathcal Z} \E[T_z(Y)^2] \, d Q(z)$ has efficient influence function
    $$ \widetilde \phi^Q_1 (x,y) =
    \int_{\mathcal Z}  T_z(y)^2\, d Q(z) - \phi^Q_1(P), $$
    \item $\phi^Q_2(P) = \displaystyle \int_{\mathcal Z} \mathbb E[T_z(Y)]^2 \, d Q(z)$ has efficient influence function
    $$ \widetilde \phi^Q_2 (x,y) = 2 \int_{\mathcal Z}  \mathbb E[T_z(Y)]  T_z(y)\, d Q(z)  -2\phi^Q_2(P).  $$
\end{itemize}

When $Q$ is unknown, the parameters 
\begin{itemize}
    \item $\Psi(P \! \otimes \! Q) $ has efficient influence function
    $ \widetilde \Psi  (x,y,z)  = \widetilde \psi  (x,y) + \E[q(z|X)^2] -  \psi^Q(P),  $
    \item $\Phi_1 (P\! \otimes \! Q) $ has efficient influence function
    $ \widetilde \Phi_1 (x,y,z) 
    =\widetilde \phi_1 (x,y) +  \mathbb E[T_z(Y)^2]  -  \phi^Q_1(P), $
    \item $\Phi_2(P\! \otimes \! Q) $ has efficient influence function
    $ \widetilde \Phi_2 (x,y,z) = \widetilde \phi_2 (x,y) +  \mathbb E[T_z(Y)]^2 - \phi^Q_2(P).  $
\end{itemize}

\section{Asymptotically efficient estimators}\label{sec:assymp_eff_estim}

\subsection{Asymptotically efficient estimators for Sobol' indices}\label{sec:assymp_eff_estim_sobol}

\paragraph{Pick-Freeze setting} In \cite{janon2012asymptotic}, the authors proposed an asymptotically efficient estimator for Sobol' indices. See Equation (6) and Proposition 2.5 in \cite{janon2012asymptotic}.

\paragraph{Given-Data setting}


In \cite{doksum1995nonparametric}, the authors considered a truncated version of $\psi(P)=\E[\E[Y|X]^2]$.  To estimate $\psi(P)$, they first estimate the regression function $m$ by a kernel estimator and then use a one-step procedure to improve the corresponding  plug-in estimator.
For a one-dimensional input $X$, an asymptotically efficient estimator of $\psi(P)$ that relies on a preliminary kernel estimator of the input's density was given in \cite{DaVeiga13}, while a simpler alternative approach based on ordered statistics can be found in \cite{klein2024efficiency}. More recently, combining the approaches of   \cite{doksum1995nonparametric} and mirror transformations (see \cite{bertin2020adaptive} and \cite{pujol2022nonparametric}), asymptotically efficient estimators of $\psi(P)$ are provided in \cite{GKLPdV21} for an input $X$ of any dimension, under adequate regularity conditions.

\subsection{Asymptotically efficient estimators for (generalized) Cramér-von Mises indices}\label{sec:assymp_eff_estim_cvm}

\paragraph{Pick-Freeze setting} 

For (generalized) Cramér-von Mises indices, one can build an asymptotically efficient estimators of all three parameters if $Q$ is known, under mild regularity conditions. Indeed, for a fixed $z \in \mathbb R^k$, the estimators 
$$ \psi_z(P) = F_{Y, Y'}(z, z) \quad \text{ and } \quad \phi_z(P) = F_Y(z) $$
are efficiently estimated by their empirical versions (using all data available for the second one),
$$  \widehat \psi_{z,n} = \frac 1 n \sum_{i=1}^n \ind_{Y_i \leq z , Y_i' \leq z}  \quad \text{ and } \quad \widehat \phi_{z,n} = \frac 1 {2n} \sum_{i=1}^n \big( \ind_{Y_i \leq z} + \ind_{ Y_i' \leq z} \big). $$
These are known to be asymptotically efficient by Lemma \ref{lem:eff_expansion} (see also \cite[Lemma 2.6]{janon2012asymptotic}) as equal to their first order expansion with the efficient influence function provided in the proof of Proposition \ref{prop:EIF_PF_mu_known}. Note that, as a result, $\phi_z^2(P)$ is also efficiently estimated by $\widehat \phi_{z,n}^2$, by virtue of Equations \eqref{eq:EIF_compo} and \eqref{eq:estim_compo}. \\

Integrated over a known measure $Q$ preserves the asymptotic efficiency:
$$ \widehat \psi^Q_n = \int \widehat \psi_{z,n} d Q(z) 
= \frac 1 n \sum_{i=1}^n \big( 1 - F_Z^{-}(Y_i \vee Y'_i) \big) = \psi^Q(P) + \frac 1 n \sum_{i=1}^n \widetilde \psi^Q( Y_i , Y'_i) $$
where $\widetilde \psi^Q$ is given in Proposition \ref{prop:EIF_PF_mu_known}. The same goes for the other two parameters. \\

If $Q$ is unknown but an i.i.d.\ sample $Z_1,\cdots,Z_n$ drawn from $Q$ is available, the empirical estimator for the parameter $P \! \otimes \! Q \mapsto \Psi(P \! \otimes \! Q)= \psi^Q(P)$ (denoted by $\Psi( F_{Y,Y'} , F_{Z})$ by abuse of notation) becomes 
$$ \widehat \Psi_n = \Psi(\widehat F_{Y,Y'} , \widehat F_{Z}) = \int \widehat F_{Y,Y'} (z,z) d \widehat F_{Z} (z) $$
where 
$$ \widehat F_{Y,Y'}(y,y') = \frac 1 n \sum_{j=1}^n \ind_{Y_j \leq y, Y'_j \leq y' } \quad \text{ and } \quad \widehat F_{Z}(z) = \frac 1 n \sum_{j=1}^n \ind_{Z_j \leq z } \quad , \quad y,y',z \in \mathbb R^k  $$
denote the empirical c.d.f.s for $(Y,Y')$ and $Z$ respectively. Since the map $\Psi(. \, , \, .)$
is bi-linear, the first order expansion of the empirical estimator satisfies
\begin{equation}\label{eq:dl_cvm}  
\Psi(\widehat F_{Y,Y'} , \widehat F_{Z}) = \Psi(F_{Y,Y'} , F_{Z}) +  \Psi(\widehat F_{Y,Y'} - F_{Y,Y'} , F_{Z})  +  \Psi(F_{Y,Y'} , \widehat F_{Z} - F_Z) + o_{\P}(1/\sqrt n),
\end{equation} 
where the left out term $\Psi(\widehat F_{Y,Y'} - F_{Y,Y'} , \widehat F_{Z} - F_Z) = o_{\P}(1/\sqrt n)$ by Donsker's theorem \cite{billingsley2008probability}. It remains to notice that
$$ \Psi(\widehat F_{Y,Y'} - F_{Y,Y'} , F_{Z}) = \frac 1 n \sum_{i=1}^n \big( 1 - F_Z^-(Y_i \vee Y'_i)  \big)- \psi^Q(P) $$
while
$$ \Psi(F_{Y,Y'} , \widehat F_{Z} - \widehat F_{Z}) = \frac 1 n \sum_{i=1}^n F_{Y,Y'}(Z_i,Z_i) - \psi^Q(P) $$
to deduce that
$$ \widehat \Psi_n = \psi^Q(P) + \frac 1 n \sum_{i=1}^n \widetilde \Psi(Y_i, Y'_i, Z_i) + o_{\P}(1/\sqrt n)  $$
where the efficient influence function $\widetilde \Psi(y,y',z) =  1 - F_Z^-(y \vee y') + F_{Y,Y'} (z,z) - 2\psi^Q(P)  $ is given in Proposition \ref{prop:EIF_PF_mu_unknown}.

\paragraph{Given-Data setting}

Asymptotic efficiency is achieved under strong consistency conditions for a preliminary estimator of the conditional probability function $p(. | X)$, such as those proposed in \cite{doksum1995nonparametric} and \cite{GKLPdV21} for the conditional expectation function, as already discussed in the case of Sobol' indices. Assuming that an estimator $\widehat p(. | X)$ is available, so that 
$$ \widehat \psi_{z,n} = \frac 1 n \sum_{i=1}^n  \widehat p(z | X_i) \big( 2 \ind_{Y_i \leq z} -  \widehat p(z | X_i) \big)  $$
is an asymptotically efficient estimator of $\psi_z(P) = \mathbb E [ p^2(Y \leq z | X) ]$ for all $z \in F$, an additional technical assumption might be needed to conserve the asymptotic efficiency of the integrated estimator $\widehat \psi^Q_{n} = \int \widehat \psi_{z,n} d Q(z)$.
Indeed if $\widehat \psi_{z,n}$ satisfies \eqref{eq:ass_eff} of Lemma \ref{lem:eff_expansion} the $o_{\P}( n^{-1/2})$
 terms now depend on $z$ and one has to check that integrated those terms in $z$ provides a $o_{\P}( n^{-1/2})$ term. The same remark holds for Cram\'er-von Mises indices.

\subsection{Asymptotically efficient estimators for Cramér-von Mises indices} 

\paragraph{Pick-Freeze setting} 

For the Cramér-von Mises index where $Q = P_Y$, one verifies by the decomposition of Equation \eqref{eq:dl_cvm} that the estimator
$$ \widehat \Psi_n = \frac 1 2 \big( \Psi(\widehat F_{Y,Y'} , \widehat F_{Y}) + \Psi(\widehat F_{Y,Y'} , \widehat F_{Y'}) \big) $$
satisfies the expression of Lemma \ref{lem:eff_expansion} with the efficient influence function given in Proposition \ref{prop:EIF_PF_mu_Y}. It is thus asymptotically efficient.

\bibliographystyle{abbrv}
\bibliography{biblio_SA}

\end{document}